\documentclass[11pt]{amsart}

\usepackage[latin1]{inputenc}
\usepackage{amsmath,amsthm,amssymb}
\usepackage[bookmarks,colorlinks]{hyperref}
\usepackage{color}
\usepackage{enumerate}
\usepackage{algorithmic}
\usepackage[boxed]{algorithm}
\usepackage{mathrsfs}
\usepackage[all]{xy}
\usepackage{tikz}
\usetikzlibrary{arrows,shapes,positioning,decorations,automata,petri}
\usepgflibrary{arrows} 
\usepgflibrary{decorations.markings}
\usepackage{ulem}
\usetikzlibrary{arrows.meta}

\numberwithin{equation}{section}

\newcommand{\Spec}{\textnormal{Spec}\,}
\newcommand{\supp}{\mathrm{supp}}
\newcommand{\Ht}{\mathrm{Ht}}
\newcommand{\OI}{\mathcal{O}}
\newcommand{\Tn}{\mathbb{T}}
\newcommand{\ind}{\mathrm{ind}}
\newcommand{\RS}{\mathcal J}

\newcommand{\mult}{\mathcal T}
\newcommand{\tail}{\mathcal L}
\newcommand{\head}{\mathcal H}
\newcommand{\PR}{R}
\newcommand{\Pomm}{\mathcal P}
\newcommand{\PO}{\Pomm_\OI}
\newcommand\rid[1]{\rightarrow_{#1}}
\newcommand{\ridc}[1]{\rightarrow^+_{#1}}
\newcommand{\cone}{\mathcal C}
\newcommand{\MS}[1]{\underline{\mathcal Ms}_{#1}}
\newcommand{\MB}[1]{\underline{\mathcal Mb}_{#1}}
\newcommand{\MSsch}[1]{\mathcal Ms_{#1}}
\newcommand{\MBsch}[1]{\mathcal Mb_{#1}}
\newcommand{\Hilb}[2]{\mathrm{Hilb}^{#1}(\mathbb A_K^{#2})}

\newtheorem{lemma}{Lemma}[section]
\newtheorem{theorem}[lemma]{Theorem}
\newtheorem{corollary}[lemma]{Corollary}
\newtheorem{proposition}[lemma]{Proposition}

\theoremstyle{definition}

\newtheorem{remark}[lemma]{Remark}
\newtheorem{example}[lemma]{Example}
\newtheorem{definition}[lemma]{Definition}
\DeclareMathAlphabet{\mathpzc}{OT1}{pzc}{m}{it}

\begin{document}

\title{The close relation between border and Pommaret marked bases} 

\author[C.Bertone]{Cristina Bertone}
\address{Dipartimento di Matematica dell'Universit\`{a} di Torino\\ 
Via Carlo Alberto 10, 10123 Torino, Italy}
\email{{cristina.bertone@unito.it}}

\author[F.Cioffi]{Francesca Cioffi}
\address{Dipartimento di Matematica e Applicazioni dell'Universit\`{a} di Napoli Federico II\\   
via Cintia, 80126  Napoli, Italy}
\email{{francesca.cioffi@unina.it}}

\normalem



\begin{abstract}
Given a finite order ideal $\mathcal O$ in the polynomial ring $K[x_1,\dots, x_n]$ over a field $K$, let $\partial \OI$ be the border of $\OI$ and $\PO$ the Pommaret basis of the ideal generated by the terms outside $\OI$.
In the framework of reduction structures introduced by Ceria, Mora, Roggero in 2019, we investigate relations among $\partial\OI$-marked sets (resp.~bases) and $\PO$-marked sets (resp.~bases). 

We prove that a $\partial\OI$-marked set $B$ is a marked basis  if and only if the $\PO$-marked set $P$ contained in $B$ is a marked basis and generates the same ideal as $B$. 
Using a functorial description of these marked bases, as a byproduct we obtain that the affine schemes respectively parameterizing $\partial\OI$-marked bases and $\PO$-marked bases are isomorphic. We are able to describe this isomorphism as a projection that can be explicitly constructed without the use of Gr\"obner elimination techniques. In particular, we obtain a straightforward embedding of border schemes in  smaller affine spaces. Furthermore, we observe that Pommaret marked schemes give an open covering of punctual Hilbert schemes. Several examples are given along all the paper. 
\end{abstract} 

\maketitle

\section*{Introduction}

Let $K$ be a field and $\PR_A:=A[x_1,\dots, x_n]$ the polynomial ring over a Noetherian $K$-algebra $A$ in the variables $x_1<\dots < x_n$. Let $\Tn$ denote the set of terms, i.e.~monic monomials, in $R_A$.

Given an order ideal $\OI\subseteq \Tn$, the ideals $I\subset \PR_A$ such that $\OI$ is an $A$-basis of $\PR_A/I$ have been extensively investigated and characterized in literature,  for instance  because they are suitable tools for the study of Hilbert schemes. These ideals can be identified by means of particular sets of generators called \emph{marked bases}. 
 Probably the most popular marked bases are  Gr\"obner bases, which need a term order, but they are not particularly suitable for studying Hilbert schemes, because in general they provide   locally closed subsets covering a Hilbert scheme instead of open subsets \cite[Theorem 6.3 i)]{LR2011}. Two types of term order free marked bases   proved to be  more suitable to investigate Hilbert schemes parameterizing $0$-dimensional schemes
 (for example, see \cite{  BCRAffine,  BLR, HM2011, HM2017, KR2008, KR2011,MS2005}  and the references therein). In this paper we focus on these latter ones.

If $\OI$ is a {\em finite} order ideal, then also the \emph{border} $\partial\OI$ is finite and the ideal generated by the terms outside $\OI$ admit a Pommaret basis $\PO$, which is contained in $\partial\OI$. Hence, in this paper we consider \emph{finite} order ideals and focus on the known characterizations of the ideals $I\subset \PR_A$, such that $\OI$ is an $A$-basis of $\PR_A/I$, that have been obtained by using either \emph{border bases} ($\partial\OI$-marked bases in this paper) or \emph{marked bases over a quasi-stable ideal} ($\PO$-marked bases in this paper). We explicitly describe a close relation between these types of marked bases.

Given a finite order ideal $\OI$,  $\partial\OI$-marked sets and bases are made of {\em monic marked} polynomials \cite{RStu}, whose head terms form $\partial\OI$.
Border bases were first introduced in \cite{MMM}, in the context of Gr\"obner bases, for computing a minimal basis of an ideal of polynomials vanishing at a set of rational points, also using duality. Border bases were then investigated from a numerical point of view because of their stability with respect to perturbation of the coefficients \cite{MS, Stetter} and have also attracted  interest from an algebraic point of view (see \cite{KKR}, \cite[Section 6.4]{KRvol2}). Furthermore, given a finite order ideal $\mathcal O$, the $\partial\OI$-marked bases parameterize an open subset  of a punctual Hilbert scheme (e.g.~\cite{HM2011, Lederer}).
This fact was used, for instance, to investigate elementary components of punctual Hilbert schemes in \cite{HM2017}. As $\OI$ varies in the set of order ideals of a prescribed cardinality, the corresponding border schemes give an open cover of a punctual Hilbert scheme \cite[Remark 3.2  items b) and e)]{KR2008}, \cite[Remark 2.8]{KR2011} and \cite[Proposition~1]{Lederer}.

Recall that every Artinian monomial ideal in $\PR_K$ has a \emph{Pommaret basis}. So, given a finite order ideal $\OI$, we can consider the
\emph{Pommaret basis} $\PO$ of the monomial ideal generated by $\Tn\setminus\OI$ and construct $\PO$-marked sets and bases, i.e.~sets and bases made of {\em monic marked} polynomials whose head terms form $\PO$.  We recall that every strongly stable ideal, also in the non-Artinian case (i.e.~if the order ideal $\OI$ is not finite), has a finite Pommaret basis. Hence, working on strongly stable ideals, $\PO$-marked bases were first introduced in \cite{CR} and investigated in \cite{BLR} with the aim to parameterize open subsets of a Hilbert scheme and study it locally. We highlight  that $\PO$-marked bases do not need any finiteness assumption on the underlying order ideal, nor they need a term order, even if $\PO$-marked bases have some nice properties similar to those of Gr\"obner bases (see for instance \cite[Theorems 3.5 and 4.10]{BCRAffine}). The
$\PO$-marked bases were considered in \cite{BCLR,CMR2015} in the case of homogeneous polynomials. In \cite{BCRAffine} non-homogeneous $\PO$-marked bases were studied, obtaining more efficient computational techniques than those in the homogeneous case.  An affine scheme parameterizing $\PO$-marked bases has been already described and used, for instance in \cite{Almost, BCRGore}, to successfully investigate Hilbert schemes. These schemes give an open cover of Hilbert schemes for any dimension, up to suitable gruoup actions.  

The motivating question of our work is: what is the relation between $\partial\OI$-marked sets (and bases) and $\PO$-marked sets (and bases) for a given finite order ideal $\OI$?
We answer this question, and as a byproduct we also establish the relation between the schemes paramaterizing these marked bases, giving also a computational Gr\"obner free method to eliminate some variables from the equations defining one in order to obtain  the other.

We use the framework of reduction structures introduced in \cite{CMR} and a functorial approach in order to compare the schemes parameterizing these two different types of bases. We prove that there is a close relation among $\partial\OI$-marked bases and $\PO$-marked bases (see Theorem \ref{thm:Pmarkedscheme}) and observe that the affine schemes parameterizing the ideals generated by these two types of bases are isomorphic, also giving an explicit isomorphism (see Corollary \ref{cor:isomorfismo}, Theorem \ref{thm:elimination} and Corollary~\ref{cor:isom}). We are able to describe this isomorphism as a projection that can be explicitly constructed without the use of Gr\"obner elimination techniques, obtaining an embedding of border schemes in affine spaces of lower dimension than the one where they are naturally embedded. Moreover, as we have already pointed out, we have an open cover of punctual Hilbert schemes made of marked schemes, without the use of any group actions which are instead needed for Hilbert schemes parameterizing schemes of positive dimension (see Proposition \ref{prop:cover}). Several examples are exhibited along all the paper.

The paper is organized in the following way.

In Section \ref{sec:genset} we recall some general notations and facts, the framework of reduction structures introduced in \cite{CMR} and some known results for the Pommaret reduction structure.

In Section  \ref{sec:redstr} we focus on $\partial\OI$-marked bases. 
Observing that a set $B$ of marked polynomials on $\partial \mathcal O$ always contains a set $P$ of marked polynomials on   $\mathcal P_{\mathcal O}$, we prove that $B$ is a $\partial \mathcal O$-marked basis if and only if $P$ is a $\mathcal P_{\mathcal O}$-marked basis and generates the same ideal as $B$ (Theorem \ref{prop:BorderAndPommB}).
We also compare the border reduction structure implicitly considered in \cite{KRvol2} with the one we give in Definition \ref{def:BRS}. In particular, we relate the \emph{border division algorithm} of \cite[Proposition 6.4.11]{KRvol2} to the reduction relation induced by the border reduction structure, where the terms of the border are ordered according to the degree.
In this setting, we can prove a Buchberger's criterion for $\partial\OI$-marked bases (Proposition \ref{prop:Spolivarmult}), which is alternative to that of  \cite[Proposition~6.4.34]{KRvol2}.

In Section \ref{sec:functors}, using a functorial approach, we easily prove that the scheme parameterizing $\partial\mathcal O$-marked bases, called \emph{border marked scheme}, and the scheme parameterizing $\mathcal P_{\mathcal O}$-marked bases, called \emph{Pommaret marked scheme}, are isomorphic (Corollary \ref{cor:isomorfismo}).
 The monicity of the marked sets we consider is crucial for the use of functors. In fact, although
in \cite{CMR} the authors deal with the polynomial ring $\PR_K$, everything  works also in $\PR_A$ thanks  to the monicity of marked polynomials and sets. In Proposition \ref{prop:cover}, as a byproduct of our investigation we easily obtain an open cover of a punctual Hilbert schemes by marked schemes.

In Section \ref{sec:computations}, we explicitly exhibit an isomorphism between the border marked scheme and the Pommaret marked scheme using Theorem \ref{prop:BorderAndPommB}. As a byproduct, we prove that there is always a subset of the variables involved in the generators of the ideal defining the border marked scheme that can be eliminated (see Corollary \ref{cor:elim}), obtaining in this way the ideal defining the Pommaret marked scheme. This elimination does not use Gr\"obner elimination techniques, which are in theory useful, but practically impossible to use (see Example \ref{ex:final}). Corollary \ref{cor:isom} is the geometric version of Corollary \ref{cor:elim}. Upper bounds for the degrees of the polynomials involved in this computation are given in  Proposition \ref{prop:degree}.

\section{Generalities and setting}
\label{sec:genset}

Let $K$ be a field, $\PR:=K[x_1,\dots,x_n]$ the polynomial ring in $n$ variables with coefficients in $K$, $\Tn$ the set of terms in $\PR$. 
If $Y$ is a subset of $\{x_1,\dots, x_n\}$, we denote by $\mathbb T[Y]$ the subset of $\Tn$ containing only terms in the variables in $Y$.
We denote by $A$ a Noetherian $K$-algebra with unit $1_K$ and set $\PR_A:=A\otimes_K \PR$.  
When it is needed, we assume  $x_1<\dots<x_n$ (this is just an ordering on the variables, this is {\em not} a term order). For every term $\tau$ in $\Tn$ we denote by $\min(\tau)$ the smallest variable appearing in $\tau$ with non-null exponent.

\begin{definition}\label{def:qstable}
A monomial ideal $J\subset \PR$ is \emph{quasi-stable} if, for every term $\tau \in J$ and for every $x_i>\min(\tau)$, there is a positive integer $s$ such that $x_i^s\tau/\min(\tau)$ belongs to $J$.
\end{definition}

It is well known that a quasi-stable ideal $J$ has a special monomial generating set that is called \emph{Pommaret basis} and has a very important role in this paper. The terms in the Pommaret basis of $J$ can be easily detected thanks to the following property (see \cite[Definition 4.1 and Proposition 4.7]{CMR2015}). For every term $\sigma \in J$,
\begin{equation}\label{eq:starset}
\sigma \text{ belongs to the Pommaret basis of }J \Leftrightarrow \frac{\sigma}{\min(\sigma)}\notin J.
\end{equation}

\begin{definition}
A set $\OI$ of terms in $\Tn$ is called an \emph{order ideal} if
$$\forall \ \sigma \in \Tn, \forall \ \tau \in\OI, \ \sigma \mid \tau \ \Rightarrow \sigma \in \OI.$$
\end{definition}

Given a monomial ideal $J\subseteq R$, the set of terms outside $J$ is an order ideal. On the other hand, if $\OI$ is an order ideal, then $J$ is the monomial ideal such that $J\cap\Tn=\Tn\setminus \OI$.

Along the paper we only consider \emph{finite} order ideals $\OI$. This is equivalent to the fact that the Krull dimension of the quotient ring $R/J$ is zero, i.e~$R/J$ is Artinian. In this case, $J$ is quasi-stable \cite[Corollary 2.3]{BCRAffine}. So, given a finite order ideal $\OI$, we can always consider the Pommaret basis of the ideal generated by $\Tn\setminus \OI$ in $\PR$ and we denote it by~$\PO$. 

\begin{definition} \label{def:border} 
{\cite[Definition 6.4.4, Proposition 6.4.6]{KRvol2}}
Given a finite order ideal $\OI$, \emph{the border} of $\OI$ is
\[
\partial \OI:=\{x_i\cdot \tau \ \vert \ \tau \in \OI, i\in\{1,\dots,n\}\}\setminus \OI.
\]
Since $\OI \cup \partial \OI$ is an order ideal too, for every integer $k\geq 0$ we can also define the {\em $k$-th border} $\partial^k\OI$ of $\OI$ in the following way:
$$\partial^0\OI:=\OI \text{ and } \partial^k \OI:= \partial(\partial^{k-1}\OI\cup\dots \cup \partial^0 \OI).$$
For every $\mu\in \Tn\setminus \OI$, the integer $\ind_\OI(\mu):=\min\{k \ \vert \ \mu \in \partial^k\OI\}$ is called the {\em index} of $\mu$ with respect to $\OI$.
\end{definition}

For every finite order ideal $\OI\subset \Tn$, it is immediate to observe that $\PO$ is contained in $\partial \OI$, thanks to  \eqref{eq:starset}. 

\begin{example}\label{ex:bordo-Pomm}
If we consider $\PR=K[x_1,x_2]$ and $\OI=\{1,x_1,x_2,x_1x_2\}$, then the Pommaret basis of the ideal $(\Tn\setminus \OI)$ is
$\PO=\{x_1^2,x_1^2x_2,x_2^2\}$ and the border of $\OI$ is $\partial\OI=\{x_1^2,x_1^2x_2,x_1x_2^2,x_2^2\}$.
\end{example}

Given a subset $T\subseteq \mathbb T$ of terms, we denote by $\langle T\rangle_A$ the module generated by $T$ over $A$. Moreover, for every term $\alpha$, the set $\cone_T(\alpha):=\{\tau \alpha\,\vert\, \tau \in T\}$ is called {\em cone of $\alpha$ by  $T$}. 

\begin{definition}\cite[Definition 3.1]{CMR}
A {\em reduction structure} $\mathcal J$ (RS, for short) in $\Tn$ is a $3$-uple $\RS:=(\head,\tail:=\{\tail_\alpha\, \vert\, \alpha \in \head\},\mult:=\{\mult_\alpha\, \vert\, \alpha \in \head\})$ where:
\begin{itemize}
\item $\head\subseteq \Tn$ is a {\em finite set} of terms;
\item for every $\alpha\in \head$,  $\mult_\alpha\subseteq \Tn$ is an order ideal, called the {\em multiplicative set of $\alpha$}, such that $\cup_{\alpha \in \head}\cone_{\mult_\alpha}(\alpha)=(\head)$;
\item  for every $\alpha\in \head$, $\tail_\alpha$ is a finite subset of $\Tn\setminus\cone_{\mult_\alpha}(\alpha)$ called the {\em tail set of $\alpha$}.
\end{itemize}
Given a reduction structure $\RS=(\head, \tail,\mult)$, we say that: $\RS$ has \emph{maximal cones} if, for every $\alpha\in \head$, $\mult_\alpha=\Tn$; $\RS$ has \emph{disjoint cones} if, for every $\alpha,\alpha' \in \head$, $\cone_{\mult_\alpha}(\alpha)\cap \cone_{\mult_{\alpha'}}(\alpha')=\emptyset$; $\RS$ has \emph{multiplicative variables} if, for every $\alpha \in \head$,  there is $Y_\alpha\subset \{x_1,\dots, x_n\}$ such that $\mult_\alpha=\Tn[Y_\alpha]$.
\end{definition}

Recall that the {\em support} of a polynomial $f\in \PR_A$ is the finite subset $\supp(f)\subseteq \Tn$ of those terms that appear with non-null coefficients in $f$.

\begin{definition} 
\cite{RStu} A {\em marked polynomial} is a polynomial $f\in \PR_A$ together with a specified term of $\supp(f)$ which appears in $f$ with coefficient $1_K$. It  will be called {\em head term of $f$} and denoted by $\Ht(f)$.
\end{definition} 

\begin{definition} \cite[Definitions 4.2 and 4.3]{CMR}
Given a reduction structure $\RS=(\head, \tail, \mult)$, a set $F$ of exactly $\vert \head\vert$ marked polynomials in $\PR_A$ is called an {\em $\head$-marked set} if, for every $\alpha\in \head$, there is $f_\alpha\in  F$ with $\Ht(f_\alpha)=\alpha$ and $\supp(f)\setminus\{\alpha\}\subseteq \tail_\alpha$.

Let $\OI_\head$ be the order ideal given by the terms of $\Tn$ outside the ideal generated by $\head$.
An $\head$-marked set $ F$ is called an {\em $\head$-marked basis} if $\OI_\head$ is a free set of generators for $\PR_A /(F)$ as $A$-module, that is $(F) \oplus \langle \OI_\head\rangle_A =\PR_A$. 
\end{definition}

Given a reduction structure $\RS=(\head, \tail, \mult)$ and an  $\head$-marked set $F$, a {\em $(\head)$-reduced form mo\-du\-lo~$(F)$ of a polynomial $g\in \PR_A$} is a polynomial $h\in \PR_A$ such that $g-h$ belongs to  $( F)$ and $\supp(h)\subseteq \OI_\head$. If there exists a unique $(\head)$-reduced form modulo $( F)$ of $g$ then it is called {\em $(\head)$-normal form modulo $(F)$ of $g$} and is denoted by $\mathrm{Nf}(g)$.

In the following, when 
$\OI_\head$ is finite, we will write $\OI_\head$-reduced form (resp.~$\OI_\head$-normal form) instead of $(\head)$-reduced form (resp.~$(\head)$-normal form). Furthermore, for any order ideal $\OI$, if $h$ belongs to $\langle \OI\rangle_A$, then we say that $h$ is {\em $\OI$-reduced}.

\begin{remark} \label{rem:crucial}
Given a reduction structure $\RS=(\head, \tail, \mult)$,  if $F$ is an $\head$-marked basis, then every polynomial $g\in \PR_A$ admits the $(\head)$-normal form $\mathrm{Nf}(g)$ modulo $(F)$. This is a direct consequence of the fact that $\PR_A$ decomposes in the direct sum $(F)\oplus \langle \OI_\head\rangle_A$, so that for every $g\in \PR_A$ there is a unique writing $g=h+h'$ with $h\in (F)$ and $h'\in \langle \OI_\head\rangle_A$. Hence,  $h'$ is the $(\head)$-normal form of $g$ modulo $(F)$.
\end{remark}

The definition of a reduction relation over polynomials can be useful to computationally detect  reduced and normal forms. 

\begin{definition} \label{def:reduction relation} \cite[page 105]{CMR} 
Given a reduction structure $\RS=(\head, \tail, \mult)$ and an  $\head$-marked set $F$, the {\em reduction relation} associated to $F$ is the transitive closure $\ridc{F\, \RS}$ of the relation on $\PR_A$ that is defined in the following way.  For $g,h\in \PR_A$, we say that $g$ is in relation with $h$ and write $g \rid{F\, \RS} h$ if there are terms $\gamma \in \supp(g)$ and $\alpha\in \head$ such that $\gamma=\alpha\eta$  belongs to  $\cone_{\mult_\alpha}(\alpha)$ and $h=g-c\eta f_\alpha$, where $c$ is the coefficient of $\gamma$ in $g$.
\end{definition}

Given a reduction structure $\RS=(\head, \tail, \mult)$ and an  $\head$-marked set $F$, the reduction relation $\ridc{ F\, \RS}$ is {\em Noetherian} if there is no infinite reduction chain $g_1 \rid{ F\, \RS} g_2 \rid{ F\, \RS} \dots$. 
The reduction structure $\RS$ is said Noetherian if $\ridc{ F\, \RS}$ is Noetherian for every  $\head$-marked set~$F$ (see \cite[Section~5]{CMR}).

Moreover, $\ridc{ F\, \RS}$ is {\em confluent} if for every polynomial $g\in \PR_A$ there exists only one $(\head)$-reduced form $h$ modulo $(F)$  such that $g \ridc{F\, \RS} h$.  The reduction structure $\RS$ is confluent if $\ridc{\mathcal F\, \RS}$ is confluent for every  $\head$-marked set $F$ (see \cite[Definition 7.1]{CMR}).

\begin{remark} \label{rem:crucial2}
If $\ridc{F\, \RS}$ is a reduction relation such that every polynomial $g\in \PR_A$ admits a $(\head)$-reduced form modulo $(F)$ then $(F) + \langle \OI_\head\rangle_A =\PR_A$. This holds in particular if $\ridc{F\, \RS}$ is Noetherian and confluent.
\end{remark}

\begin{definition}
Given a finite order ideal $\OI$, the \emph{Pommaret reduction structure} $\RS_{\PO}=(\mathcal H,\mathcal L,\mathcal T)$ is the reduction structure with $\head:=\PO$ and, for every $\alpha \in \PO$, $\tail_\alpha:=\OI$ and $\mult_\alpha:=\Tn\cap K[x_1,\dots,\min(\alpha)]$.
\end{definition}

The Pommaret reduction structure was investigated in \cite[Definition 4.2]{BCRAffine} for  an arbitrary order ideal~$\OI$, not necessarily finite. The reduction relation $\ridc{ F\,\RS_{\PO}}$ is Noetherian and confluent for every $\PO$-marked set $F$ \cite[Proposition~4.3]{BCRAffine}. 
 Then, it is well-known that the Pommaret reduction relation can be used to characterize $\mathcal P_{\OI}$-marked bases in the following way.

\begin{definition}
Given two marked polynomials $f,f'$ with $\Ht(f)=\alpha$ and $\Ht(f')=\alpha'$,  the {\em $S$-polynomial} of $(f,f')$ is 
$S(f,f')=\eta f-{\eta'}f'$, where $\eta \Ht(f)={\eta'}\Ht(f')=lcm(\Ht(f),\Ht(f'))$.
\end{definition}

\begin{definition}\label{def:nonMultc}
Let $\OI$ be a finite order ideal and $\RS=(\head,\tail,\mult)$ a reduction structure such that $(\head)=(\Tn\setminus\OI)$.
For every $\alpha,\alpha'\in \head$, 
 $(\alpha,\alpha')$ is a {\em non-multiplicative couple} if  there is $x_i\notin \mult_{\alpha}$ such that $x_i\alpha \in \cone_{\mult_{\alpha'}}(\alpha')$.
\end{definition}

\begin{proposition} {\rm (Buchberger's criterion for Pommaret marked bases)} \cite[Proposition 5.6]{BCRAffine}\label{prop:Buchpomm}
Let $\OI$ be a finite order ideal, consider the Pommaret reduction structure $\RS_{\PO}$ and let $P$ be a $\PO$-marked set.
The following are equivalent:
\begin{enumerate}
\item $P$ is a $\PO$-marked basis; 
\item \label{it:Snonmult} for every $p, p'\in P$ such that $(\Ht(p),\Ht(p'))$ is a non-multiplicative couple, $S(p,p')\ridc{P\,\RS_{\PO}}~0$.
\end{enumerate}
\end{proposition}

\section{Border reduction structure}
\label{sec:redstr}

In this section, given a finite order ideal $\OI$, we focus on  reduction structures such that $\head$ is the border of $\OI$  and their relations with the Pommaret reduction structure.

\begin{definition}\cite[Table 1 and Lemma 13.2]{CMR}\label{def:BRS}
Given a finite order ideal $\OI$, let  the terms of the border $\partial \OI$ be ordered in an arbitrary way and labeled coherently, i.e.~for every $\beta_i,\beta_j\in \partial \OI$,   if $i<j$ then $\beta_i$ precedes $\beta_j$. 

The \emph{border reduction structure} $\RS_{\partial\OI}:=(\mathcal H,\mathcal L,\mathcal T)$ is the reduction structure with $\head=\partial \OI$ and, for every $\beta_i\in \partial \OI$, $\tail_{\beta_i}:=\OI$ and $\mult_{\beta_i}:=\{\mu\in\Tn \ \vert \ \forall  j>i, \,\beta_j\text{ does not divide }\beta_i\mu\}$.
\end{definition}

\begin{remark}\label{rem:otherBSs1}
Given a finite order ideal $\OI$, 
in~\cite{KRvol2} the authors consider the reduction structure $\RS'=(\head',\tail',\mult')$ with $\head'=\partial \OI$ (ordered arbitrarily and labeled coeherently as in $\RS_{\partial\OI}$), and $\tail'_\beta=\OI$, $\mult'_\beta=\Tn$ for every $ \beta\in \partial \OI$.
The border reduction structure $\RS_{\partial\OI}$ given in Definition \ref{def:BRS} is a \emph{substructure} of $\RS'$, because  the multiplicative sets of $\RS_{\partial\OI}$ are contained in those of $\RS'$ (see \cite[Definition~3.4]{CMR} for the definition of substructure).
\end{remark} 

Like observed in \cite[Remark 4.6]{CMR}, the definition of marked \emph{basis} only depends on the ideal generated by the marked set we are considering, and  this notion does not rely on the reduction relation associated to the reduction structure we are considering. However, in order to prove that the notions of $\partial \OI$-marked basis and $\PO$-marked basis associated to $\RS_{\partial\OI}$ and $\RS_{\PO}$, respectively, are closely related, we need the features of the  reduction relation given by the Pommaret reduction structure.

\begin{theorem}\label{prop:BorderAndPommB}
Let $\OI$ be a finite order ideal, $\RS_{\partial\OI}$ be the border reduction structure, with terms of $\partial \OI$ ordered arbitrarily, and $\RS_{\PO}$ be the Pommaret reduction structure. Let $B$ be a $\partial \OI$-marked set in $\PR_A$, $P$ the $\PO$-marked set contained in $B$ and $B'=B\setminus P$. Then,
\[
B \text{ is a }\partial \OI \text{-marked basis}\Leftrightarrow P \text{ is a } \PO\text{-marked basis and } B'\subset (P).
\]
\end{theorem}

\begin{proof} 
If $B$ is a $\partial \OI$-marked basis, then $\PR_A=(B)\oplus \langle \OI\rangle_A$ and, in particular, $(B)\cap \langle \OI\rangle_A=\{0\}$. Since the $\PO$-marked set $P$ is a subset of $B$, we also have $(P)\cap \langle \OI\rangle_A=\{0\}$.  Furthermore, since the Pommaret reduction structure is Noetherian and confluent, every polynomial in $R_A$ has a $\OI$-reduced form modulo $(P)$, that is $(P) + \langle \OI\rangle_A=R_A$ by Remark \ref{rem:crucial2}. These two facts together imply that $P$ is a marked basis.

Moreover, for every polynomial $b'\in B'$, let $\tilde b'\in \langle \OI\rangle_A$ be its $\OI$-normal form modulo $(P)$, hence $b'-\tilde b'$ belongs to $(P)\subset (B)$. This means that $\tilde b'$ is also the only $\OI$-normal form modulo $(B)$, which is equal to $0$ by hypothesis (see Remark \ref{rem:crucial}). Hence, $b'$ belongs to $(P)$.

Assume now that $P$ is a $\PO$-marked basis and $B'\subset (P)$. Observing that in the present hypotheses $(P)\oplus \langle \OI\rangle_A=\PR_A$ and $(P)=(P\cup B')=(B)$, we immediately conclude.
\end{proof}

\begin{example}\label{ex:dopo teorema}
Consider the finite order ideal $\OI=\{1,x_1,x_2,x_1x_2\}\subset \PR=K[x_1,x_2]$ as in Example \ref{ex:bordo-Pomm}, the $\partial \OI$-marked set $B_{e_1,e_2}$ given by the following polynomials
\[
b_{1}:=x_{{1}}^{2}-x_{{1}}-x_{{2}}-1,\, b_2:=x_{{2}}^{2}-3\,x_{{2}},\, b_3:=x_{{1}}^{2
}x_{{2}}-x_{{1}}x_{{2}}-4\,x_{{2}},\, b_4:=x_{{1}}x_{{2}}^{2}- e_1 x_{{1}}x_{{
2}}-e_2 \,x_{{1}}
, \, e_1,e_2 \in K,
\] 
and the $\PO$-marked set $P=\{b_1, b_2, b_3\}$.

For every $e_1,e_2 \in K$, $P$ is a $\PO$-marked basis, by Proposition \ref{prop:Buchpomm}. If $e_1\neq 3$ or $e_2\neq 0$, then $b_4$ does not belong to~$(P)$, hence, by Theorem \ref{prop:BorderAndPommB}, $B_{e_1,e_2}$ is not a $\partial\OI$-marked basis for these values of $e_1$ and $e_2$. 
If $e_1=3$ and $e_2=0$ then $b_4$ belongs to $(P)$, hence $B_{3,0}$ is a $\partial \OI$-marked basis.
\end{example}

In \cite{KRvol2}, the authors consider a procedure called \emph{border rewrite relation} \cite[page 432]{KRvol2} which is the reduction relation of the reduction structure $\RS'$ presented in Remark~\ref{rem:otherBSs1}. The reduction structure $\RS'$ is not Noetherian, as highlighted in  \cite[Example 6.4.26]{KRvol2}. This is due to the fact that $\RS'$ has maximal cones and, in general, there is no term order with respect to which, for every $f$ in a $\partial \OI$-marked set $F$ and for every $\gamma$ in $\OI$, the head term $\Ht(f)$ of $f$ is bigger than $\gamma$ (see \cite{RStu} and \cite[Theorem 5.10]{CMR}). 

In general, given a reduction structure $\RS$,  disjoint cones are not sufficient to ensure Noetherianity. For instance,  also the substructure $\RS_{\partial\OI}$ of $\RS'$ is in general not Noetherian. We highlight this rephrasing the above quoted example of \cite{KRvol2} for the reduction structure $\RS_{\partial \OI}$.

\begin{example}\cite[{Example 6.4.26}]{KRvol2}
Consider the order ideal $\OI=\{1,x_1,x_2,x_1^2,x_2^2\}\subset \PR_K=K[x_1,x_2]$. The border of $\OI$ id $\partial \OI=\{x_1x_2, x_1^3,x_1^2x_2,x_1x_2^2,x_2^3\}$. We consider the border reduction structure $\RS_{\partial\OI}$ with the terms of the border ordered in the following way:
\[\beta_1= x_1^3,\,\beta_2=x_1^2x_2,\,\beta_3=x_1x_2^2,\,\beta_4=x_2^3,\,\beta_5=x_1x_2,\]
and the $\partial \OI$-marked set $B$ given by the marked polynomials
\[
b_i=\beta_i \text{ for } i=1,\dots,4, \quad b_5=\beta_5-x_1^2-x_2^2.
\]
The reduction structure $\RS_{\partial \OI}$ has disjoint cones, however this does not imply Noetherianity. Indeed, consider the term $\gamma=x_1^2x_2^2$. When we reduce $\gamma$ by $\rid{B\,\RS_{\partial\OI}}$, we  fall in the same infinite loop as using the border rewrite relation of \cite{KRvol2}. Recalling that  a term $\gamma \in (\partial\OI)$ is reduced using the term $\beta_i$ with $i=\max\{j\ \vert \ \beta_j\in\partial \OI\text{ divides }\gamma\}$, we obtain
\begin{multline*}
x_1x_2^2\rid{B\, \RS_{\partial \OI}}x_1x_2^2-x_2b_5=x_1^2x_2+x_2^3\rid{B\,\RS_{\partial\OI}}x_1^2x_2+x_2^3-b_4=x_1^2x_2\rid{B\,\RS_{\partial\OI}}\\
\rid{B\,\RS_{\partial\OI}}x_1^2x_2-x_1b_5=x_1^3+x_1x_2^2\rid{B\,\RS_{\partial\OI}}x_1^3+x_1x_2^2-b_1=x_1x_2^2.
\end{multline*}
\end{example}

{\em In the following, we will order the terms of $\partial \OI$ either increasingly by degree (terms of the same degree are ordered arbitrarily), or increasingly according to a term order $\prec$}. In the first case, we will denote the reduction structure with $\RS_{\partial\OI,\deg}$ and, in the second case, with $\RS_{\partial\OI,\prec}$.

For every term order $\prec$, the reduction structure $\RS_{\partial \OI,\prec}$ is Noetherian \cite[page~127]{CMR}. 
We now focus on properties of the reduction relation of a border reduction structure $\RS_{\partial\OI,\deg}$ and explicitly prove that $\RS_{\partial\OI,\deg}$ is Noetherian and confluent.  
We use properties of the border described in \cite[Section 6.4.A]{KRvol2} and results in \cite{CMR}.

\begin{lemma}\label{lem:IndB}
Let $\OI$ be a finite order ideal.
\begin{enumerate}
\item\label{it:IndB1} If $\gamma$ is a term in $(\Tn\setminus \OI)$ and $\beta \in \partial \OI$ is a term of maximal degree dividing $\gamma$, then $\ind_{\OI}(\gamma)=\deg(\gamma/\beta)+1$;
\item \label{it:IndB2} if $\delta \in \OI$ and $\eta \in \Tn$ are terms such that  $\delta\eta$ belongs to $(\Tn\setminus \OI)$ and if $\beta \in \partial \OI$ is a term of maximal degree dividing  $\delta\eta$, then $ \deg(\eta)>\deg(\delta\eta/\beta)$.
\end{enumerate}
\end{lemma}

\begin{proof}
Item \eqref{it:IndB1} is a direct consequence of \cite[Proposition~6.4.8 a)]{KRvol2}.
For what concerns item~\eqref{it:IndB2}, we have $\ind_\OI(\delta\eta)=\deg(\delta\eta/\beta)+1$ by \eqref{it:IndB1} and conclude, because $\ind_\OI(\delta\eta)\leq\deg(\eta)$ by \cite[Proposition~6.4.8~b)]{KRvol2}.  
\end{proof}

\begin{proposition}
Given a finite order ideal $\OI$, consider the border reduction structure $\RS_{\partial\OI,\deg}$ and a $\partial\OI$-marked set $B$. Then, the reduction relation $\ridc{ B\,\RS_{\partial\OI,\deg}}$ is Noetherian and confluent. 
\end{proposition}

\begin{proof}
First we prove that $\RS_{\partial \OI,\deg}$ has disjoint cones,  that is $\cone_{\mult_ {\beta_i}}(\beta_i)\cap\cone_{\mult_ {\beta_j}}(\beta_j)=\emptyset$, for every $\beta_i,\beta_j \in\partial \OI$. Indeed, assume $j>i$ and consider $\gamma\in\cone_{\mult_ {\beta_i}}(\beta_i)$. By definition of  cone, there is $\eta\in \mult_{\beta_i}$ such that $\eta\beta_i=\gamma$ and for every $j>i$, $\beta_j$ does not divide $\gamma$. Hence, $\gamma$ does not belong to $\cone_{\mult_ {\beta_j}}(\beta_j)$. 

Furthermore, $\RS_{\partial\OI,\deg}$ is Noetherian. It is sufficient to consider $\Tn$ with the well founded order given by the index of a term with respect to $\OI$, and apply \cite[Theorem 5.9]{CMR}. Indeed, consider $f,g \in \PR_A$ such that $f\rid{B\,\RS_{\partial\OI,\deg}}g$ and let $\gamma=\eta\beta_i \in \supp(f)\cap  \cone_{\mult_ {\beta_i}}(\beta_i)$ be the term which is reduced, i.e.~$f-c\eta  b_i=g$, with $c\in R_A$ the coefficient of $\gamma$ in $f$ and $ b_i$ the polynomial of $B$ such that $\Ht( b_i)=\beta_i$.
Consider a term $\gamma'$ in $\supp(g)\setminus \supp(f)$.  By  construction we obtain that $\gamma'$ is divisible by $\eta$. So, if $\gamma'$ belongs to $\cone_{\mult_{\beta_j}}(\beta_j)$ for some $\beta_j\in \partial \OI$, then by Lemma \ref{lem:IndB} $\ind_\OI(\gamma')=\deg(\gamma'/\beta_j)+1<\deg(\eta)+1=\ind_\OI(\gamma)$. Hence, $\ridc{B\,\RS_{\partial\OI,\deg}}$  is Noetherian. 

Summing up,  $\RS_{\partial \OI,\deg}$ has disjoint cones and it is Noetherian. Hence, we can conclude that $\ridc{B\,\RS_{\partial\OI,\deg}}$ is also confluent by \cite[Remark~7.2]{CMR}.
\end{proof}

We highlight that, given  the border reduction structure $\RS_{\partial \OI,\deg}$ and a $\partial\OI$-marked set $B$, the reduction relation $\ridc{F\,\RS_{\partial\OI,\deg}}$ is equivalent to the border division algorithm of \cite[Proposition 6.4.11]{KRvol2}, in the sense that the $\OI$-reduced forms obtained by the reduction relation $\ridc{F\,\RS_{\partial\OI,\deg}}$ are {\em normal $\OI$-reminders} of the border division algorithm (see \cite[page~426]{KRvol2}).

Recall that for a given finite order ideal $\OI$, every $\partial \OI$-marked set contains a $\PO$-marked set. Nevertheless, the different reduction relations applied on the same polynomial in general give different $\OI$-reduced  forms.

\begin{example}\label{ex:diffRforms}
Consider the order ideal $\OI=\{1,x_1,x_2,x_1x_2\}\subset \PR=K[x_1,x_2]$, as in Example~\ref{ex:bordo-Pomm}, and the border reduction structure $\RS_{\partial \OI,\deg}$ and the Pommaret reduction structure $\RS_{\PO}$.  Let $B\subset \PR$ be the $\partial\OI$-marked set given by the following polynomials:
\[
b_1:=x_{{1}}^{2}-1-x_{{1}}-x_{{2}},\, b_2:=x_{{2}}^{2}-3x_{{2}},\, b_3:=x_{{1}}^{2
}x_{{2}}-4x_2-x_{{1}}x_{{2}},\, b_4:=x_{{1}}x_{{2}}^{2}-x_1-x_1x_2
\]
and $P$ the $\PO$-marked set given by the polynomials $b_1,b_2,b_3$.
The $\partial \OI$-marked set $B$ is obtained from that of Example \ref{ex:dopo teorema} by replacing both $e_1$ and $e_2$ with $1$.
Consider the polynomial $f=x_1^2x_2^2$. We compute an $\OI$-reduced form modulo $(B)$ of $f$ in the following way: 
\[
f \ridc{B\, \RS_{\partial\OI}} f-x_1 b_4= x_1^2x_2+x_1^2 \ridc{B\, \RS_{\partial\OI}}x_1^2x_2+x_1^2-b_3-b_1=x_1x_2+x_1+5x_2+1 \in \langle \OI\rangle_K.
\]
We also compute an $\OI$-reduced form modulo $(P)$ in the following way: 
\[
f \ridc{P\, \RS_{\PO}} f-x^2_1b_2= 3x_1^2x_2 \ridc{P\, \RS_{\PO}}3x_1^2x_2-3 b_3=3x_1x_2+12x_2 \langle \OI\rangle_K.
\]
\end{example}

We now focus on how the reduction relations corresponding to border reduction structures are used in order to obtain marked bases.

\begin{definition} \label{def:ncouple}
For every $\alpha,\alpha'$ terms in   a finite set $M\subset \Tn$, $(\alpha,\alpha')$ is a \emph{neighbour couple} if either $\alpha=x_j\alpha'$ for some variable $x_j$ or $x_i \alpha=x_j\alpha'$ for some couple of variables $(x_i,x_j)$.
\end{definition}

The definition of neighbour couple is due to \cite[Definition~6.4.33~c)]{KRvol2} in the framework of border bases, however Definition \ref{def:ncouple} is given for any  finite set of terms, not necessarily the border of an order ideal. Furthermore, the notion of neighbour couple does not depend on the reduction structure we are considering.

Given a finite order ideal $\OI$ and recalling Definition \ref{def:nonMultc},  if we consider the Pommaret reduction structure $\RS_{\PO}$, then a couple of terms in $\PO$ can be simultaneously non-multiplicative and neighbour, although this is not always the case (see Example \ref{ex:NumSpoli}).

We now investigate the non-multiplicative couples for  the border reduction structure $\RS_{\partial\OI,\deg}$, showing that non-multiplicative couples are always neighbour couples in this case.

\begin{lemma}\label{lem:nonmultBorder}
Given a finite order ideal $\OI$, consider the border reduction structure $\RS_{\partial\OI,\deg}$, and let $(\beta_i,\beta_j)$ be a non-multiplicative couple of $\partial \OI$,  i.e. there is $x_\ell \notin \mult_{\beta_i}$ such that $x_\ell\beta_i\in \cone_{\mult_{\beta_j}}(\beta_j)$. More precisely, $x_\ell\beta_i=\delta \beta_j$ for some $x_\ell\notin \mult_{\beta_i}$ and $\delta \in \mult_{\beta_j}$.   
Then $j>i$ and 
$\deg(\delta)\leq 1$.
\end{lemma}

\begin{proof}
By definition of $\mult_{\beta_j}$ for a border reduction structure (Definition \ref{def:BRS}), the equality $x_\ell\beta_i=\delta \beta_j$, with $x_l\notin \mult_{\beta_i}$ and $\delta \in \mult_{\beta_j}$, implies that  $j>i$. Furthermore $\deg(\beta_i)\leq \deg(\beta_j)$, since in $\RS_{\partial \OI,\deg}$ the terms of $\delta \OI$ are ordered according to increasing degree. This implies that $\deg(\delta)\leq \deg(x_\ell)=1$.
\end{proof}

As already observed in Remark \ref{rem:otherBSs1}, $\RS_{\partial \OI,\deg}$ is a substructure of $\RS'$. So,  we can rephrase Buchberger's criterion for border marked bases given in \cite[Proposition 6.4.34]{KRvol2} in terms of the reduction relation given by $\RS_{\partial \OI,\deg}$.

\begin{proposition} \rm{(Buchberger's criterion for  border marked bases)} \cite[Proposition 6.4.34]{KRvol2}\label{prop:Buchborder} 
Let $\OI$ 
be a finite order ideal, and consider the reduction structure $\RS_{\partial\OI,\deg}$.
Let $B$  be a $\partial \OI$-marked set. 
The following are equivalent:
\begin{enumerate}
\item \label{it:Sneigh1}$B$ is a $\partial \OI$-marked basis; 
\item \label{it:Sneigh3}for every $(b,b')$ such $(\Ht(b),\Ht(b'))$ is a  neighbour couple, $S(b,b')\ridc{B\,\RS_{\partial\OI,\deg}} 0$.
\end{enumerate}
\end{proposition}

Applying \cite[Theorem 11.6]{CMR} to the reduction structure $\RS_{\partial\OI,\prec}$, for some term order $\prec$, an alternative Buchberger's criterion can be obtained. However, this criterion might involve couples of terms in $\partial \OI$ which are neither neighbour, nor  non-multiplicative, unless  $\RS_{\partial\OI,\prec}$ has multiplicative variables.  For instance,  if $\prec_{lex}$ is the lex term order, then the reduction structure $\RS_{\partial\OI,\prec_{lex}}$ has multiplicative variables (see \cite[Theorem 13.5]{CMR}). 

In general, the reduction structure $\RS_{\partial\OI,\deg}$ does not have multiplicative variables, not even if we consider  $\RS_{\partial\OI,\prec}$ for a degree compatible term order \cite[Example 13.4]{CMR}. Nevertheless, we now show that  it is sufficient to consider non-multiplicative couples of terms $\partial \OI$ in $\RS_{\partial \OI,\deg}$ in order to obtain another Buchberger's criterion for border bases. In terms  of continuous involutive division, this result also follows from results by Riquier and Janet (see \cite{Janet}, \cite[Definitions 3.1 and 4.9, Theorem 6.5]{GB}). For the sake of completeness, we give a complete proof using the notions adopted in the present paper. More explicitely, we prove the equivalence of the following Proposition to Proposition \ref{prop:Buchborder}.

\begin{proposition}\label{prop:Spolivarmult}
Let $\OI$ be a finite order ideal, and consider the reduction structure $\RS_{\partial\OI,\deg}$.
Let $B$  be a $\partial \OI$-marked set.
The following are equivalent:
\begin{enumerate}
\item \label{it:Sneigh2}$B$ is a $\partial \OI$-marked basis; 
\item \label{it:Sneigh4}  for every $(b,b')$ such $(\Ht(b),\Ht(b'))$ is a non-multiplicative couple, $S(b,b')\ridc{B\,\RS_{\partial\OI,\deg}} 0$.
\end{enumerate}
\end{proposition}

\begin{proof}
Here, we prove the equivalence between Proposition \ref{prop:Buchborder}\eqref{it:Sneigh3} and item \eqref{it:Sneigh4}. By Lemma~\ref{lem:nonmultBorder}, one implication is immediate, so we move to proving that item $\eqref{it:Sneigh4} $ implies Proposition~\ref{prop:Buchborder}\eqref{it:Sneigh3}. 

Consider now a neighbour couple $(\beta_i,\beta_j)$ which is not a non-multiplicative couple. Since $\partial \OI$ is increasingly ordered by degree, if $x_s\beta_i=\beta_j$ then $j>i$, hence  $x_s\notin \mult_{\beta_i}$ by definition of $\mult$ as given in Definition \ref{def:BRS}. Furthermore, if $x_s\beta_i=x_l\beta_j$, then it is not possible that $x_s$ belongs to $\mult_{\beta_i}$ and simultaneously $x_\ell\in \mult_{\beta_j}$, again by Definition \ref{def:BRS}.

Hence, the only case of neighbour couple which is not non-multiplicative is 
$x_s\beta_i=x_l\beta_j$ with $x_s\notin \mult_{\beta_i}$ and $x_\ell\notin \mult_{\beta_j}$. Let $\beta_k$ be the unique term in $\partial \OI$ such that $x_s\beta_i=x_\ell\beta_j=\delta \beta_k$ with $\delta \in \mult_{\beta_k}$. Hence both $(\beta_i,\beta_k)$ and $(\beta_j,\beta_k)$ are non-multiplicative couples and by Lemma  \ref{lem:nonmultBorder} we obtain that $i,j<k$ and $\deg(\delta)\leq1$. Let $b_i,b_j,b_k\in B$ be the marked polynomials such that $\Ht(b_i)=\beta_i$, $\Ht(b_j)=\beta_j$, and $\Ht(b_k)=\beta_k$.
Since $S(b_i,b_k)$ and $S(b_j,b_k)$ reduce to 0 by~$\ridc{B\,\RS_{\partial\OI,\deg}}$, then also $S(b_i,b_j)$ reduces to 0 by $\ridc{B\,\RS_{\partial\OI,\deg}}$, because $S(b_i,b_j)=S(b_i,b_k)-S(b_j,b_k)$ and $\ridc{B\,\RS_{\partial\OI,\deg}}$ is confluent (see also \cite[Proposition 4.5]{KR2008}).
\end{proof}

\begin{remark}
The most used effective criterion to check whether a $\partial \OI$-marked set is a basis is not a Buchberger's criterion. The commutativity of formal multiplication matrices is usually preferred (see \cite[Theorem 3.1]{Mourrain1999}). This is simply due to the fact that the border reduction given in \cite{KRvol2} is neither Noetherian nor  confluent, hence when in a statement one finds \lq\lq $f\ridc{B\,\RS_{\partial\OI}} 0$\rq\rq, one should read \lq\lq it is possible to find a border reduction path leading $f$ to 0\rq\rq. 
\end{remark}

\begin{example}\cite[Example 4.6]{KR2008}\label{ex:NumSpoli}
Consider the finite order ideal $\OI=\{1,x_1,x_2,x_3,x_2x_3\}\subset \PR=K[x_1,x_2,x_3]$, so that  $\partial \OI:=\{b_1,\dots,b_8\}$ with $b_1=x_1^2$, $b_2=x_1x_2$, $b_3=x_1x_3$, $b_4=x_2^2$, $b_5=x_3^2$, $b_6=x_1x_2x_3$, $b_7=x_2^2x_3$, and $b_8=x_2x_3^2$, while $\PO=\partial\OI\setminus \{b_8\}=\{b_1,\dots,b_7\}$.  In this case the reduction structure $\RS_{\partial\OI,\deg}$ has multiplicative variables.

Inspired by Ufnarovski graphs \cite[Definition 48.1.1]{SPES4}, similarly to what is done in \cite[Definition 8.5]{Stetter}, in Figures \ref{fig:ExampleFig1}, \ref{fig:ExampleFig2}, \ref{fig:ExampleFig3} we represent graphs whose vertices are the terms in $\partial \OI$ (we use bullets for the terms in $\PO$ and a star for the unique term in $\partial\OI\setminus \PO$), and whose edges are given  by either neighbour couples of terms in $\partial \OI$ (Figure \ref{fig:ExampleFig1}), or non-multiplicative couples of terms $\partial \OI$ in $\RS_{\partial \OI,\deg}$ (Figure \ref{fig:ExampleFig2}) or non-multiplicative couples of terms $\PO$ in $\RS_{\PO}$ (Figure \ref{fig:ExampleFig3}). In Figures \ref{fig:ExampleFig2} and \ref{fig:ExampleFig3}, we use arrows for edges; the arrow starts from $b_i$ and ends at $b_j$ if $(b_i,b_j)$ is a non-multiplicative couple with $x_\ell b_i=\delta b_j$. 

\begin{figure}[h]
\begin{tikzpicture}[decoration={
markings,
mark=between positions 0 and 1 step 4mm with {\arrow{stealth}}},
p/.style={circle,fill,
inner sep=0pt,minimum size=1.5mm},
b/.style={star,fill,
inner sep=0pt,minimum size=3mm}]
\node(A) at ( 3,3) [p, label=above:$b_1$] {};
\node (B) at ( 1.85,1.7) [p,label=left:$b_3$] {}; 
\node (C) at ( 4.15,1.7) [p,label=right:$b_2$] {}; 
\node (H) at ( 3,1)  [p,label=above:$b_6$]{};
\node (D) at ( 0.5,0) [p,label=below left:$b_5$] {};
\node (E) at ( 2,0) [b,label=below:$b_8$] {};
\node (F) at ( 4,0)  [p,label=below right:$b_7$]{};
\node (G) at (5.5,0)  [p,label=below right:$b_4$]{};
\draw (A) -- (B);
\draw (A) -- (C);
\draw (B) -- (H);
\draw (C) -- (H);
\draw (B) -- (D);
\draw (D) -- (E);
\draw (E) -- (H);
\draw (E) -- (F);
\draw (F) -- (H);
\draw (F) -- (G);
\draw (C) -- (G);
\draw (B) -- (C);
\end{tikzpicture}
\caption{Neighbour couples for $\RS_{\partial\OI,\deg}$}
\label{fig:ExampleFig1}
\end{figure}

\begin{figure}[h]
\begin{minipage}[b]{0.48\linewidth}
\centering
\begin{tikzpicture}[decoration={
markings,
mark=between positions 0 and 1 step 4mm with {\arrow{stealth}}},
p/.style={circle,fill,
inner sep=0pt,minimum size=1.5mm},
b/.style={star,fill,
inner sep=0pt,minimum size=3mm}]
\node(A) at ( 3,3) [p, label=above:$b_1$] {};
\node (B) at ( 1.85,1.7) [p,label=left:$b_3$] {};
\node (C) at ( 4.15,1.7) [p,label=right:$b_2$] {};
\node (H) at ( 3,1)  [p,label=above:$b_6$]{};
\node (D) at ( 0.5,0) [p,label=below left:$b_5$] {};
\node (E) at ( 2,0) [b,label=below:$b_8$] {};
\node (F) at ( 4,0)  [p,label=below right:$b_7$]{};
\node (G) at (5.5,0)  [p,label=below right:$b_4$]{};
\draw  [-{Latex[length=3mm]}](A) -- (B);
\draw  [-{Latex[length=3mm]}](A) -- (C);
\draw [-{Latex[length=3mm]}] (B) -- (H);
\draw  [-{Latex[length=3mm]}](C) -- (H);
\draw [-{Latex[length=3mm]}] (B) -- (D);
\draw  [-{Latex[length=3mm]}](D) -- (E);
\draw  [-{Latex[length=3mm]}](H) -- (E);
\draw  [-{Latex[length=3mm]}](F) -- (E);
\draw  [-{Latex[length=3mm]}](H) -- (F);
\draw [-{Latex[length=3mm]}] (G) -- (F);
\draw [-{Latex[length=3mm]}] (C) -- (G);
\end{tikzpicture}
\caption{Non-multiplicative couples for $\RS_{\partial\OI,\deg}$}
\label{fig:ExampleFig2}
\end{minipage}
\begin{minipage}[b]{0.48\linewidth}
\centering
\begin{tikzpicture}[decoration={
markings,
mark=between positions 0 and 1 step 4mm with {\arrow{stealth}}},
p/.style={circle,fill,
inner sep=0pt,minimum size=1.5mm},
b/.style={star,fill,
inner sep=0pt,minimum size=3mm}]
\node(A) at ( 3,3) [p, label=above:$b_1$] {};
\node (B) at ( 1.85,1.7)  [p,label=left:$b_3$] {};
\node (C) at ( 4.15,1.7) [p,label=right:$b_2$] {};
\node (H) at ( 3,1)  [p,label=above:$b_6$]{};
\node (D) at ( 0.5,0) [p,label=below left:$b_5$] {};
\node (E) at ( 2,0) [b,label=below:$b_8$] {};
\node (F) at ( 4,0)  [p,label=below right:$b_7$]{};
\node (G) at (5.5,0)  [p,label=below right:$b_4$]{};
\draw [-{Latex[length=3mm]}](A) -- (B);
\draw [-{Latex[length=3mm]}](A) -- (C);
\draw [-{Latex[length=3mm]}](B) -- (H);
\draw[-{Latex[length=3mm]}] (C) -- (H);
\draw[-{Latex[length=3mm]}] (B) -- (D);
\draw[-{Latex[length=3mm]}] (H)--(D);
\draw[-{Latex[length=3mm]}] (F)  to [out=220, in=320] (D);
\draw[-{Latex[length=3mm]}] (H) -- (F);
\draw [-{Latex[length=3mm]}](G) -- (F);
\draw[-{Latex[length=3mm]}] (C) -- (G);
\end{tikzpicture}
\caption{Non-multiplicative couples for $\RS_{\PO}$}
\label{fig:ExampleFig3}
\end{minipage}
\end{figure}
\end{example}

\begin{remark}
Observe that the non-multiplicative couples of terms in $\PO$ of $\RS_{\PO}$ which are not neighbour  couples of terms in $\partial \OI$ of $\RS_{\partial\OI,\deg}$ are such that $x_\ell b_i=\delta b_j$ with $\deg(\delta)=s>1$. 
In this case, there are  $b_{i_2},\dots b_{i_{s-1}}\in \partial \OI\setminus \PO$ such that $(b_i,b_{i_2})$, $(b_{i_{s-1}},b_j)$ and , for every $t\geq 2$, $(b_{i_t},b_{i_{t+1}})$  are  neighbour couples of terms in $\partial \OI$ of $\RS_{\partial\OI,\deg}$. In Example \ref{ex:NumSpoli}, this is the case for the non-multiplicative couples $(b_6,b_5)$ and $(b_7,b_5)$ of terms in $\PO$ of $\RS_{\PO}$.\\
This means that in general there are less $S$-polynomials to consider in Proposition \ref{prop:Buchpomm}\eqref{it:Snonmult} than those in Proposition \ref{prop:Buchborder}\eqref{it:Sneigh3} or Proposition \ref{prop:Spolivarmult}\eqref{it:Sneigh4}.
\end{remark}

\section{Border and Pommaret marked functors and their representing schemes}\label{sec:functors}

\emph{From now, for what concerns border reduction structures, given a finite order ideal $\OI$, we only consider the border reduction structure $\RS_{\partial \OI,\deg}$, which will be simply denoted by $\RS_{\partial\OI}$.}

Given a finite order ideal $\OI$, the goal of this section is proving that the two affine schemes, which respectively parameterize $\partial \OI$-marked bases and  $\PO$-marked bases, are isomorphic. 

\begin{definition}\cite[Appendix A]{CMR}\label{def:func}
Let $\OI$ be a finite order ideal and $\RS=(\head, \tail, \mult)$ a reduction structure with $(\head)=(\Tn\setminus \OI)$. The functor
\[
\MS{\RS}: \text{ Noeth-}k\text{-Alg }\longrightarrow \text{ Sets},
\]
associates to every Noetherian $k$-Algebra $A$ the set $\MS{\RS}(A)$ consisting of all the  ideals $I\subset \PR_A$ generated by an $\head$-marked set, and to every morphism of Noetherian $k$-algebras $\phi: A\rightarrow A'$ the morphism $\MS{\RS}(\phi):\MS{\RS}(A)\rightarrow \MS{\RS}(A')$ that operates in the following natural way:
\[
\MS{\RS}(\phi)(I)=I\otimes_A A'.
\]
The following subfunctor of $\MS{\RS}$
\[
\MB{\RS}: \text{ Noeth-}k\text{-Alg }\longrightarrow \text{ Sets},
\]
associates to every Noetherian $k$-Algebra $A$ the set $\MB{\RS}(A)$ consisting of all the ideals $I\subset \PR_A$ generated by an $\head$-marked basis, and to every morphism of Noetherian $k$-algebras $\phi: A\rightarrow A'$ the morphism $\MB{\RS}(\phi):\MB{\RS}(A)\rightarrow \MB{\RS}(A')$ that operates in the following natural way:
\[
\MB{\RS}(\phi)(I)=I\otimes_A A'.
\]
\end{definition}

\begin{remark}
The monicity of marked sets and bases ensures that marked sets and bases are preserved by extension of scalars (see also \cite[Lemmas A.1 and  A.2]{CMR}).
\end{remark}

The functor $\MS{\RS}$ is the functor of points of an  affine scheme $\MSsch{\RS}$ isomorphic to $\mathbb A^N$, for a suitable integer $N$. More precisely, if $\tail_\alpha=\OI$ for every $\alpha\in \head$, then $\MS{\RS}$ is the functor of points of the affine scheme $\MSsch{\RS}\simeq \mathbb A^{\vert \head\vert \cdot \vert \OI\vert}$ \cite[Lemma A.1]{CMR}.

We now investigate $\MB{\RS_{\partial \OI}}$ and $\MB{\RS_{\PO}}$ more in detail  in order to understand the relation among them. 

\begin{definition}
Let $\OI=\{\sigma_1,\dots, \sigma_\ell\}$ be a finite order ideal and  $\RS=(\head, \tail, \mult)$ any reduction structure with $\head=\{\tau_1,\dots, \tau_m\}$ a set of generators of the ideal $(\Tn\setminus \OI)$ and $\tail_\alpha:=\OI$, for every $\alpha\in \head$. Let $C:=\{C_{ij}\vert 1\leq i\leq m,  1\leq j\leq  \ell\}$ be a new set of parameters.
The \emph{generic $\head$-marked set} is the set of marked polynomials $\{g_1,\dots,g_m\}\subset \PR_{K[C]}$ where
\[
g_i=\tau_i-\sum_{j=1}^\ell C_{ij}\sigma_j.
\]
\end{definition}

Observe that the generic $\mathcal H$-marked set is a $\mathcal H$-marked set in $\PR_A$ with $A=K[C]$.

The generic $\partial \OI$-marked set was first introduced in \cite[Definition 3.1]{KR2008}. We denote it by $\mathscr B$, and observe that, up to relabelling the parameters $C$, $\mathscr B$ always contains the generic $\PO$-marked set, that we denote by $\mathscr P$. From now on we  denote by $\mathscr B'$ the set $\mathscr B\setminus \mathscr P$ and by $\tilde C$ the set of parameters in the polynomials of $\mathscr B'$.

\begin{example}\label{ex:genMset}
We consider again the finite order ideal $\OI=\{1,x_1,x_2,x_1x_2\}$. Then, the generic $\partial \OI$-marked set $\mathscr B$ is given by the following polynomials:
\[b_1:=x_{{1}}^
{2}-C_{{1,1}}-C_{{1,2}}x_{{1}}-C_{{1,3}}x_{{2}}-C_{{1,4}}x_{{1}}x_{{2}},\quad
b_2:=x_{{2}}^{2}-C_{{2,1}}-C_{{2,2}}x_{{1}}-C_{{2,3}}x_{{2}}-C_{{2,4}}x_{{1}}x_{{2}},\]
\[b_3=x_{{1}}^{2}x_{{2}}-C_{{3,1}}-C_{{3,2}}x_{{1}}-x_{{2}}C_{{3,3}}-C_{{3,4}}x_{{1}}x_{{2}},\quad
b_4:=x_{{1}}x_{{2}}^{2}-C_{{4,1}}-C_{{4,2}}x_{{1}}-C_{{4,3}}x_{{2}}-C_{{4,4}}x_{{1}}x_{{2}}.
\]
Observe that in this case the generic $\PO$-marked set $\mathscr P$  is equal to $\{b_1,b_2,b_3\}$, $\mathscr B'$ is equal to $\{b_4\}$, $C$ consists of 16 parameters and $\tilde C=\{C_{{4,1}},C_{{4,2}},C_{{4,3}},C_{{4,4}}\}$.
\end{example}

The functor $\MB{\partial O}$ (resp. $\MB{\PO}$) is the functor of points of a closed subscheme of $\MSsch{\partial \OI}=\mathbb A^{\vert \OI\vert \cdot\vert \partial O\vert}$ (resp.  $\MSsch{\PO}=\mathbb A^{\vert \OI\vert \cdot\vert \PO\vert}$), see \cite[Proposition 4.1]{KR2008} and \cite[Proposition 3]{Lederer} (resp. \cite[Theorem 6.6]{BCRAffine}).
We now recall the construction of the two subschemes representing the functors $\MB{\partial O}$ and $\MB{\PO}$.  

Given a polynomial in $K[C][x_1,\dots,x_n]$, the coefficients  in $K[C]$ of the terms in the variables $x_1,\dots,x_n$ are called {\em $\mathbf x$-coefficients}.

\begin{definition}\label{def:Pscheme}
Let $\mathscr P=\{b_i\in \mathscr B\vert \Ht(b_i)\in \PO\}\subset \mathscr B$ be the generic $\PO$-marked set.
For every  $p_r,p_s\in \mathscr P$ such that $(\Ht(p_r),\Ht(p_s))$ is a non-multiplicative couple, consider the $\OI$-reduced polynomial $h'_{rs}\in\langle \OI\rangle_A$ such that $S(p_k,p_s)\ridc{\mathscr P\,\RS_{\PO}} h'_{rs}$. 
Let $\mathfrak P\subset K[C\setminus \tilde C]$ be the ideal generated by the $\mathbf x$-coefficients of the polynomials $h'_{rs}$. 
The affine scheme $\Spec(K[C\setminus \tilde C]/\mathfrak P)$ is the \emph{Pommaret marked scheme} of $\OI$.
\end{definition}

\begin{theorem}\cite[Theorem 6.6]{BCRAffine}\label{thm:Pmarkedscheme}
$\MB{\RS_{\PO}}$ is the functor of points of $\Spec(K[C\setminus \tilde C]/\mathfrak P)$.
\end{theorem}

\begin{definition}\label{def:Bscheme}
Let $\mathscr B=\{b_i\}_{i=1,\dots,m}$ be the generic $\partial \OI$-marked set. For every $b_r,b_s\in \mathscr B$ such that $(\Ht(b_r),\Ht(b_s))$ is a neighbour couple, consider the $\OI$-reduced polynomial  such that $S(b_r,b_s)\ridc{ \mathscr B\,\RS_{\partial\OI}} h_{rs}$. 
Let $\mathfrak B\subset K[C]$ be the ideal generated by the $\mathbf x$-coefficients of the polynomials $h_{rs}$. 
The affine scheme $\Spec(K[C]/\mathfrak B)$ is the \emph{border marked scheme} of $\OI$.
\end{definition}

The border marked scheme was introduced in \cite{KR2008}, and it was further investigated in \cite{HM2011, Lederer}. 
In \cite{Lederer} the interested reader can find a different proof of the fact that $\MB{\RS_{\partial \OI}}$ is the functor of points of the border marked  scheme. In \cite[Definition 3.1]{KR2008}, although the functor is not explicitly defined, the authors define the border marked scheme using the commutativity of formal matrices.

For the sake of completeness, we explicitly prove that the affine scheme defined by the ideal $\mathfrak B$ of Definition \ref{def:Bscheme} represents the functor $\MB{\RS_{\partial \OI}}$ just like  the schemes in \cite[Definition 3.1]{KR2008}, \cite[Proposition 3]{Lederer}. Our proof is inspired by \cite[Appendix~A]{CMR}. 

\begin{theorem}\label{thm:reprfuncB} 
$\MB{\RS_{\partial \OI}}$ is the functor of points of $\Spec(K[C]/\mathfrak B)$.
\end{theorem}

\begin{proof}
For every $K$-algebra $A$, there is a 1-1 correspondence between $\mathrm{Hom}(K[C]/\mathfrak B, A)$ and $\MB{\RS_{\partial \OI}}(A)$.
In fact, on the one hand we consider a morphism of $K$-algebras $\pi : K[C]/\mathfrak B\rightarrow A$ and extend it in the obvious way to a morphism between the polynomial rings $\PR_{K[C]/\mathfrak B}$ and $\PR_A$. Then, we  can associate to every such morphism $\pi$ the $\partial \OI$-marked basis $\pi(\mathscr B)$, where $\mathscr B$ is the generic $\partial \OI$-marked set. The ideal generated by $\pi(\mathscr B)$ belongs to $\MB{\RS_{\partial \OI}}(A)$.

On the other hand, every $\partial \OI$-marked basis $B=\{b_i=\tau_i-\sum c_{ij}\sigma_j\vert c_{ij}\in A\}_{i=1,\dots,m}\subset \PR_A$ generates an ideal belonging to $\MB{\RS_{\partial \OI}}(A)$  and defines a morphism $\pi_B: K[C]/\mathfrak B\rightarrow A$ given by $\pi_B (C_{i,j}) =c_{i,j}$, thanks to Proposition \ref{prop:Buchborder}.

This 1-1 correspondence commutes with the extension of scalars, because for every morphism $\phi : A\rightarrow A'$  and for every $\partial \OI$-marked basis $B$ we have $\phi(B) =\{\phi(b_i) =\tau_i-\sum\phi(c_{ij})\sigma_j \}$, hence $\pi_{\phi(B)} = \phi\circ \pi(B)$.
\end{proof}

The functorial approach combined with Theorem \ref{prop:BorderAndPommB} gives the following result.
 
\begin{corollary}\label{cor:isomorfismo}
For a given finite order ideal $\OI$, the border marked scheme and the Pommaret marked scheme are isomorphic.
\end{corollary}

\begin{proof}
By Theorem \ref{prop:BorderAndPommB}, for every $K$-algebra $A$ the sets $\MB{\RS_{\partial\OI}}(A)$ and $\MB{\RS_{\PO}}(A)$ are equal. Hence, by Yoneda's Lemma, $\MB{\RS_{\partial\OI}}$ and $\MB{\RS_{\PO}}$ are the functor of points of the same scheme, up to isomorphism. Hence, 
$\Spec(K[C]/\mathfrak B)$ and $\Spec(K[C\setminus \tilde C]/\mathfrak P)$ are isomorphic.
\end{proof}

In \cite{BCRAffine} and in \cite{BLR} (in the latter case, under the hypothesis that $K$ has characteristic 0), the authors present an open cover for the Hilbert scheme parameterizing $d$-dimensional schemes in $\mathbb P_K^n$ with a prescribed Hilbert polynomial; the open subsets are Pommaret marked schemes, but in order to cover the whole Hilbert scheme, the action of the general linear group $\mathrm{GL}_K(n+1)$ is needed. We now highlight that, for the Hilbert scheme parameterizing $0$-dimensional subschemes of length $m$ in $\mathbb A_K^n$, the Pommaret marked schemes cover the whole Hilbert scheme {\em without any group action}. In the next statement, Proposition \ref{prop:cover}, we also  recall  the known analogous result for border marked schemes, see \cite[Remark 3.2 items b) and e)]{KR2008}, \cite[Remark 2.8]{KR2011} and \cite[Proposition 1]{Lederer}.

\begin{proposition}\label{prop:cover}
Consider the Hilbert scheme $\Hilb{m}{n}$ that parameterizes $0$-dimensional schemes of $\mathbb A^n_K$ of length $m$. Let $\mathscr O$ be the set containing the finite order ideals $\OI$ in $\PR_K$ such that $\vert\OI\vert=m$.
Then we have the two following open covers
\[
\Hilb{m}{n}=\cup_{\OI\in \mathscr O}\MBsch{\RS_{\partial \OI}}=\cup_{\OI\in \mathscr O}\MBsch{\RS_{\PO}},
\]
where  $\MBsch{\RS_{\partial \OI}}$ (resp. $\MBsch{\RS_{\PO}}$) is the scheme defined in Definition \ref{def:Bscheme} (resp. Definition \ref{def:Pscheme}) that represents the functor $\MB{\RS_{\partial \OI}}$ (resp. the functor $\MB{\RS_{\PO}}$).
For every $\OI\in \mathscr O$, the open subset $\MBsch{\RS_{\partial \OI}}$ of $\Hilb{m}{n}$ is isomorphic to the open subset $\MBsch{\RS_{\PO}}$.
\end{proposition}

\begin{proof}
See \cite[Remark 3.2 b)]{KR2008} and \cite[Proposition 1]{Lederer} for the fact that, as $\OI$ varies in $\mathscr O$, the border marked schemes cover $\Hilb{m}{n}$. The fact that also Pommaret marked schemes cover $\Hilb{m}{n}$, as $\OI$ varies in $\mathscr O$, is a consequence of \cite[Lemma 6.12 and Proposition 6.13]{BCRAffine}, combined with the fact that for every $\OI \in \mathscr O$, the ideal generated by $\Tn\setminus \OI$ is quasi-stable, i.e. has a Pommaret basis. The isomorphism between the open subsets of the open covers is that of Corollary \ref{cor:isomorfismo}.
\end{proof}

\section{The isomorphism between the border and Pommaret marked schemes}\label{sec:computations}

In the present section we explicitly present an isomorphism between the border  and the Pommaret marked schemes of the same finite order ideal $\OI$.
We describe the algebraic relation between the ideals $\mathfrak B\subset K[C]$ and $\mathfrak P\subset K[C\setminus \tilde C]$ defining the two schemes and, as a byproduct, obtain  a constructive method to eliminate the variables $\tilde C$.

\begin{theorem}\label{thm:elimination}
Let $\OI$ a finite order ideal. Let $\mathfrak B\subset K[C]$ be the ideal defining the border marked scheme of $\OI$ and $\mathfrak P\subset K[C\setminus \tilde C]$ be the ideal defining the Pommaret marked scheme  of $\OI$. Then, there is an ideal $\mathfrak B'\subset K[C]$ generated by $\vert \tilde C\vert$ polynomials of kind $C_{ij}-q_{ij}$ with $C_{ij}\in\tilde C$ and $q_{ij}\in K[C\setminus \tilde C]$, such that
\[
\mathfrak B=\mathfrak PK[C]+\mathfrak B'.
\]
\end{theorem}

\begin{proof}
Consider the generic $\partial \OI$-marked set $\mathscr B$ and the generic $\PO$-marked set $\mathscr P$ contained in $\mathscr B$. Let $\mathscr B'=\{b'_i\}$ be the set of marked polynomials in $\mathscr B\setminus \mathscr P$  and denote by $\tilde C$ the parameters appearing in the polynomials in $\mathscr B'$.
Thanks to Theorem \ref{prop:BorderAndPommB}, we can obtain the ideal $\mathfrak B$ defining the border marked scheme in the following alternative way with respect to Definition~\ref{def:Bscheme}. 

Consider $\tau_i' \in \partial \OI\setminus \PO$ and  let $b_i'$ be the marked polynomial in $\mathscr B$ such that $\Ht(b_i')=\tau_i'$. By the Pommaret reduction relation, compute the unique $\OI$-reduced polynomial $h'_i$ such that
\begin{equation}\label{eq:ridB'}
\tau_i'\ridc{\mathscr P\,\RS_{\PO}} h_i'.
\end{equation}
Let $\mathfrak B'$ be the ideal in $K[C]$ generated by the $\mathbf x$-coefficients of $b_i'-(\tau_i'-h_i')$, for $\tau_i' \in \partial \OI\setminus \PO$. Observe that the $\mathbf x$-coefficients of $b_i'-(\tau_i'-h_i')$ are of the form $C_{ij}-q_{ij}$, where $C_{ij}\in \tilde C$   and $q_{ij} \in K[C\setminus \tilde C]$. 
By Theorem \ref{prop:BorderAndPommB}, the following equality of ideals in $K[C]$ holds:
\[
\mathfrak B=\mathfrak PK[C]+\mathfrak B'.
\]
\end{proof}

\begin{corollary} \label{cor:elim}Let $\OI$ be a finite order ideal. Let $\mathfrak B\subset K[C]$ be the ideal defining the border marked scheme of $\OI$ and $\mathfrak P\subset K[C\setminus \tilde C]$ be the ideal defining the Pommaret marked scheme  of $\OI$. Then, 
\[
\mathfrak P=\mathfrak B\cap K[C\setminus \tilde C].
\]
\end{corollary}

\begin{proof}
By Theorem \ref{thm:elimination}, we have $\mathfrak B =\mathfrak P K[C]+\mathfrak B' \subseteq K[C]$.
Since the variables $\tilde C$ only appear in the generators of $\mathfrak B'$ and not in the generators we consider for $\mathfrak P$, then we obtain the thesis.
\end{proof}

Corollary \ref{cor:elim} has the computational consequence that the parameters $\tilde C$ can be eliminated from the generators of $\mathfrak B$ using the generators of the ideal $\mathfrak B'$ as computed in the proof of Theorem \ref{thm:elimination}.  This has the following geometrical consequence on the border and the Pommaret marked schemes of $\OI$.

\begin{corollary}\label{cor:isom}
Let $\OI$ be a finite order ideal. 
There is a projection morphism  from $A^{\vert C\vert}$ to $A^{\vert C\setminus \tilde C\vert}$ such that its restriction  to the border marked scheme is an isomorphism between the border marked scheme and the Pommaret one.
\end{corollary}

\begin{proof}
With the notation introduced in the proof of Theorem \ref{thm:elimination}, 
let $\Phi$ be the $K$-algebra surjective morphism $K[C] \to K[C\setminus \tilde C]$ such that 
$$\Phi(C_{i,j}):= \left\{\begin{array}{ll} q_{i,j}\in K[C\setminus \tilde C], &\text{if } C_{i,j}\in \tilde C\\ C_{i,j}, &\text{otherwise.}\end{array}\right.$$
The morphism $\Phi$ corresponds to the morphism $\phi: \mathbb A^{\vert C\setminus \tilde C\vert} \to  \mathbb A^{\vert C\vert}$, which associates to a closed $K$-point $P:=(a_1,\dots,a_{\vert C\setminus \tilde C\vert})$ the closed $K$-point $Q:=(a_1,\dots,a_{\vert C\setminus \tilde C\vert},q_{i,j}(a_1,\dots,a_{\vert C\setminus \tilde C\vert}))$.

The kernel of $\Phi$ coincides with the ideal $\mathfrak B'$ generated by the polynomials $C_{i,j}- q_{i,j}$, as $C_{i,j}$ varies in $\tilde C$. Hence, from $\Phi$ we obtain a  $K$-algebra isomorphism $\bar{\Phi}: K[C]/\mathfrak B' \to K[C\setminus \tilde C]$ which corresponds to an isomorphism $\bar{\phi}: \mathbb A^{\vert C\setminus \tilde C\vert}=\Spec(K[C\setminus \tilde C]) \to \Spec(K[C]/\mathfrak B')$ that is the inverse of the restriction to $\Spec(K[C]/\mathfrak B')$ of the projection from $A^{\vert C\vert}$ to $\mathbb A^{\vert C\setminus \tilde C\vert} $ .  

Since $\Phi(\mathfrak B)=\mathfrak P$, from $\Phi$ we obtain a $K$-algebra isomorphism $\Psi: K[C]/\mathfrak B \to K[C\setminus \tilde C]/\mathfrak P$, which gives to the inverse of the desired projection from $\Spec(K[C]/\mathfrak B)$ to $\Spec(K[C\setminus~\tilde C]/\mathfrak P)$.
\end{proof}

Corollary \ref{cor:elim} states that it is possible to know a priori a subset $\tilde C$ of  the variables in $C$ that can be eliminated from the border marked scheme, and its proof gives an effective method which does not use any Gr\"obner elimination computation.
In next Example \ref{ex:EqsSch}, the polynomials generating $\mathfrak B'$ that allow this elimination are also generators of $\mathfrak B$, 
but this does not happen in general, although linear terms often appear in the generators of $\mathfrak B$, allowing the elimination of some variables in $C$.
Actually, the polynomials we consider to generate the ideal $\mathfrak B$ have degree $\leq 2$ \cite[Remark~3.2 item~a)]{KR2008}. For the generators of $\mathfrak P$ and for the polynomials generating $\mathfrak B'$ we have the following result.

\begin{proposition}\label{prop:degree}\cite{BCRAffine}
Let $D:=\max\{\deg(\sigma)\vert \sigma\in \OI \}+1$.
\begin{enumerate}[(i)]
\item\label{it:deg1}  The generators of $\mathfrak P$ have degree bounded from above by  
$D(D^n-1)/(D-1)$.

\item  \label{it:deg2}
For every $\beta \in \partial \OI\setminus \PO$, let $s:=\deg(\beta)$. Then the generators of $\mathfrak B'$ obtained from $\beta$ by \eqref{eq:ridB'} are polynomials in $K[C]$ of degree bounded from above by $(s-1)(D^n-1)/(D-1)$.
\end{enumerate}
\end{proposition}

\begin{proof}
Item \eqref{it:deg1} is \cite[Corollary~7.8]{BCRAffine}.

For what concerns item \eqref{it:deg2}, it is sufficient  to consider the polynomial  $h_\beta\in \langle \OI\rangle$ such that $\beta\ridc{\RS_{\PO}\,\mathcal P}h_\beta$ and apply \cite[Theorem 7.5]{BCRAffine}.
\end{proof}

So, in general the generators of $\mathfrak B'$ have degree strictly bigger than 2 and do not belong to the set of generators of $\mathfrak B$ we consider, as Example \ref{ex:final} will show.

Finally, we underline that Corollary \ref{cor:isom} gives an embedding of the border marked  scheme $\Spec(K[C]/\mathfrak B)$ in an affine space of dimension lower than $\vert C\vert$, namely of dimension $\vert C\setminus \tilde C\vert$. However,  there might be embeddings of the border marked scheme in affine spaces of even smaller  dimension, as highlighted in Example \ref{ex:EqsSch}. Nevertheless, when $\vert C\setminus \tilde C\vert\ll \vert C\vert$, in general it is not possible to obtain an embedding in an affine space of dimension with Groebner elimination or direct computations, see Example \ref{ex:final}.

\begin{example}\label{ex:EqsSch}
Consider the finite order ideal $\OI=\{1,x_1,x_2,x_1x_2\}$ and the generic $\partial \OI$-marked set $\mathscr B$ as defined in Example \ref{ex:genMset}. In this case, the ideal $\mathfrak B\subset K[C]$ of Definition \ref{def:Bscheme} is generated by the following polynomials:
\[
-C_{{1,3}}C_{{2,1}}-C_{{1,4}}C_{{4,1}}+C_{{3,1}},\quad
-C_{{1,3}}C_{{2,2}}-C_{{1,4}}C_{{4,2}}+C_{{3,2}},\quad
-C_{{1,3}}C_{{2,3}}-C_{{1,4}}C_{{4,3}}-C_{{1,1}}+C_{{3,3}},\]
\[-C_{{1,3}}C_{{2,4}}-C_{{1,4}}C_{{4,4}}-C_{{1,2}}+C_{{3,4}},\quad
-C_{{1,1}}C_{{2,2}}-C_{{2,4}}C_{{3,1}}+C_{{4,1}},\quad
-C_{{1,2}}C_{{2,2}}-C_{{2,4}}C_{{3,2}}-C_{{2,1}}+C_{{4,2}},\]
\[-C_{{1,3}}C_{{2,2}}-C_{{2,4}}C_{{3,3}}+C_{{4,3}},\quad
-C_{{1,4}}C_{{2,2}}-C_{{2,4}}C_{{3,4}}-C_{{2,3}}+C_{{4,4}},\]
\[C_{{1,1}}C_{{4,2}}-C_{{2,1}}C_{{3,3}}+C_{{3,1}}C_{{4,4}}-C_{{3,4}}C_{{4,1}},\quad
C_{{1,4}}C_{{4,2}}-C_{{2,4}}C_{{3,3}}-C_{{3,2}}+C_{{4,3}}\]
\[C_{{1,2}}C_{{4,2}}-C_{{2,2}}C_{{3,3}}+C_{{3,2}}C_{{4,4}}-C_{{3,4}}C_{{4,2}}+C_{{4,1}},\quad
C_{{1,3}}C_{{4,2}}-C_{{2,3}}C_{{3,3}}+C_{{3,3}}C_{{4,4}}-C_{{3,4}}C_{{4,3}}-C_{{3,1}}.\]
They are obtained computing the $\OI$-reduced forms modulo $\mathscr B$ by $\ridc{\mathscr B\,\partial \OI}$ of the $S$-polynomials of the neighbour couples $(b_1,b_3)$, $(b_2,b_4)$, $(b_3,b_4)$.
The ideal $\mathfrak P\subset K[C\setminus \tilde C]$ of Definition \ref{def:Pscheme} is  generated by the following polynomials
\[
-C_{{1,1}}C_{{1,4}}C_{{2,2}}-C_{{1,4}}C_{{2,4}}C_{{3,1}}-C_{{1,3}}C_{{2,1}}+C_{{3,1}},\quad
-C_{{1,3}}C_{{1,4}}C_{{2,2}}-C_{{1,4}}C_{{2,4}}C_{{3,3}}-C_{{1,3}}C_{{2,3}}-C_{{1,1}}+C_{{3,3}},
\]
\[
-C_{{1,2}}C_{{1,4}}C_{{2,2}}-C_{{1,4}}C_{{2,4}}C_{{3,2}}-C_{{1,3}}C_{{2,2}}-C_{{1,4}}C_{{2,1}}+C_{{3,2}},\]
\[-{C_{{1,4}}}^{2}C_{{2,2}}-C_{{1,4}}C_{{2,4}}C_{{3,4}}-C_{{1,3}}C_{{2,4}}-C_{{1,4}}C_{{2,3}}-C_{{1,2}}+C_{{3,4}},
\]\[
-C_{{1,1}}C_{{1,2}}C_{{2,2}}+C_{{1,1}}C_{{2,2}}C_{{3,4}}-C_{{1,1}}C_{{2,4}}C_{{3,2}}-C_{{1,4}}C_{{2,2}}C_{{3,1}}-C_{{1,1}}C_{{2,1}}+C_{{2,1}}C_{{3,3}}-C_{{2,3}}C_{{3,1}},\]
\begin{multline*}
-{C_{{1,2}}}^{2}C_{{2,2}}+C_{{1,2}}C_{{2,2}}C_{{3,4}}-C_{{1,2}}C_{{2,4}}C_{{3,2}}-C_{{1,4}}C_{{2,2}}C_{{3,2}}-C_{{1,1}}C_{{2,2}}-C_{{1,2}}C_{{2,1}}+C_{{2,1}}C_{{3,4}}+\\+C_{{2,2}}C_{{3,3}}-C_{{2,3}}C_{{3
,2}}-C_{{2,4}}C_{{3,1}},\end{multline*}
\[
-C_{{1,2}}C_{{1,3}}C_{{2,2}}+C_{{1,3}}C_{{2,2}}C_{{3,4}}-C_{{1,3}}C_{{2,4}}C_{{3,2}}-C_{{1,4}}C_{{2,2}}C_{{3,3}}-C_{{1,3}}C_{{2,1}}+C_{{3,1}},\]\[
-C_{{1,2}}C_{{1,4}}C_{{2,2}}-C_{{1,4}}C_{{2,4}}C_{{3,2}}-C_{{1,3}}C_{{2,2}}-C_{{1,4}}C_{{2,1}}+C_{{3,2}}.
\]
Among the generators of $\mathfrak B$, there are polynomials containing  $C_{4,i}$, $i=1,\dots ,4$. These variables do not appear in the generators of $\mathfrak P$, because the polynomial $b_4$ is not involved in its construction. 
The polynomials that generate $\mathfrak B'\subset K[C]$ as in Theorem \ref{thm:elimination} are
\begin{equation*}
C_{{4,1}}-C_{{1,1}}C_{{2,2}}-C_{{2,4}}C_{{3,1}},\quad C_{{4,2}}-C_{
{1,2}}C_{{2,2}}-C_{{2,4}}C_{{3,2}}-C_{{2,1}},
\end{equation*}
\begin{equation*}
C_{{4,3}}- C_{{1,3}}C_{{2,2}}-C_{{2,4}}C_{{3,3}},\quad C_{{4,4}}- C_{{1,4}}C_{{2,2}}-C_{{2,4}}C_{{3,4}}-C_{{2,3}}.
\end{equation*}
The morphism $\Phi$ of Corollary \ref{cor:isom} operates in the following way: $\Phi(C_{i,j})=C_{i,j}$ for $i=1,2,3$ and every $j$, and
\[\Phi(C_{{4,1}})= C_{{1,1}}C_{{2,2}}+C_{{2,4}}C_{{3,1}},\quad \Phi(C_{{4,2}})=C_{
{1,2}}C_{{2,2}}+C_{{2,4}}C_{{3,2}}+C_{{2,1}},\]
\[\Phi(C_{{4,3}})= C_{{1,3}}C_{{2,2}}+C_{{2,4}}C_{{3,3}},\quad \Phi(C_{{4,4}})= C_{{1,4}}C_{{2,2}}+C_{{2,4}}C_{{3,4}}+C_{{2,3}}.
\]

 The morphism $\Phi$ embeds the border marked scheme $\Spec(K[C]/\mathfrak B)\subset \mathbb A^{16}$ into an affine space of dimension $\vert C\setminus \tilde C\vert=12$.  We highlight that this is not the smallest possible embedding for this border marked scheme, which is actually isomorphic to $\mathbb A^8$, as shown by direct elimination in \cite[Example 3.8]{KR2008}.
\end{example}

\begin{example}\label{ex:final}
Consider the order ideal $\mathcal O=\{1,x_{{1}},x_{{2}},x_{{3}},x_{{1}}^{2},x_{{1}}^{3},x_{{1}}^{4}\}\subset \PR_K=K[x_1,x_2,x_3]$. The border of $\OI$ is
\[
\partial \OI=\{x_{{2}}x_{{1}},x_{{2}}^{2},x_{{3}}x_{{1}},x_{{3}}x_{{2}},x_{{3}}^
{2},x_{{2}}x_{{1}}^{2},x_{{3}}x_{{1}}^{2},x_{{2}}x_{{1}}^{3},x_{
{3}}x_{{1}}^{3},x_{{1}}^{5},x_{{2}}x_{{1}}^{4},x_{{3}}x_{{1}}^{4}
\},
\]
and the Pommaret basis of the ideal generated by $\Tn\setminus \OI$ is
\[
\PO=\{x_{{3}}^{2},x_{{3}}x_{{2}},x_{{2}}^{2},x_{{3}}x_{{1}},x_{{2}}x_{{
1}},x_{{1}}^{5}\}.
\]
Observe that, in this case, the Pommaret basis of the ideal generated by $\Tn\setminus \OI$ coincides with the minimal monomial basis, because $(\Tn\setminus \OI)$ is a \emph{stable} ideal.

For the interested reader, ancillary files containing the generic marked sets $\mathscr B$ and $\mathscr P$, the set of eliminable parameters $\tilde C$ and the generators  of the ideals $\mathfrak B$, $\mathfrak P$ and $\mathfrak B'$ are available at https://sites.google.com/view/cristinabertone/ancillary/borderpommaret

Constructing $\mathfrak B$ as in Definition \ref{def:Bscheme} and thanks to Theorem \ref{thm:reprfuncB}, the border marked scheme of $\OI$ is a closed subscheme of $\mathbb A^{\vert C\vert}=\mathbb A^{84}$, and its defining ideal is generated by a set of 126 polynomials of degree 2.

Constructing $\mathfrak P$ as in Definition \ref{def:Pscheme} and thanks to Theorem \ref{thm:Pmarkedscheme},  the Pommaret marked scheme of $\OI$ is a closed subscheme of $\mathbb A^{\vert C\setminus \tilde C\vert}=\mathbb A^{42}$, and its defining ideal is defined by a set of 56 polynomials of degree  at most 5.

The polynomials $C_{ij}-q_{ij}$, which generates the ideal $\mathfrak B'$ and allow the elimination of the 42 parameters $\tilde C$, have degrees between $2$ and $4$. In particular, there are $14$ polynomials of degree $2$ that generate $\mathfrak B'$ and also belong to the set of polynomials given in Definition \ref{def:Bscheme} generating $\mathfrak B$. Hence, $14$ of the $42$ variables in $\tilde C$ can be eliminated by some generators of $\mathfrak B$. Then among the polynomials generating $\mathfrak B'$ there are $14$ polynomials of degree $3$ and $14$ polynomials of degree $4$, which do not belong to the set of the generators of $\mathfrak B$ we consider, which have degree $2$.

So, Corollary \ref{cor:elim} has two important consequences on the elimination of variables from the border marked scheme. Firstly, it identifies $\tilde C$ as a set of eliminable variables. In the present example, if one only knows that $42$ of the variables in $C$ are eliminable, it is almost impossible to find an eliminable set of this size without 
knowing that such a set is given by the variables appearing in marked polynomials with head terms in $\partial \OI\setminus \PO$. Secondly, even if one knows that the variables in  $\tilde C$  can be eliminated, this is not feasible by a Gr\"obner  elimination. We stopped the computation of such a Gr\"obner basis on this example after one hour of computation by CoCoA 5 (see \cite{CoCoA-5}) and Maple 18 (see \cite{Maple}), while it takes few seconds to compute the polynomials generating $\mathfrak B'$ by $\ridc{\mathscr P\,\RS_{\PO}}$.
\end{example}

\section*{Acknowledgments}

The authors are members of GNSAGA (INdAM, Italy). This research is partially supported  by MIUR funds FFABR-Cioffi-2017 and by Dipartimento di Matematica, Universit\`a di Torino, in the framework of the projects \lq\lq Algebra e argomenti correlati\rq\rq \ and \lq\lq Algebra, Geometria e Calcolo Numerico\rq\rq.

\bibliographystyle{amsplain}
\providecommand{\bysame}{\leavevmode\hbox to3em{\hrulefill}\thinspace}
\providecommand{\MR}{\relax\ifhmode\unskip\space\fi MR }
\providecommand{\MRhref}[2]{%
  \href{http://www.ams.org/mathscinet-getitem?mr=#1}{#2}
}
\providecommand{\href}[2]{#2}

\end{document}